\documentclass[12pt]{amsart}
\usepackage{amsmath, amsthm}
\usepackage[english]{babel}
\usepackage[T1]{fontenc}
\usepackage[latin1]{inputenc}
\usepackage{amssymb}
\usepackage{amscd}
\usepackage{latexsym}
\usepackage[bookmarksnumbered,plainpages,hypertex]{hyperref}
\usepackage{graphicx}
\usepackage{mathrsfs} 

\newtheoremstyle{theorem}
  {12pt}          
  {12pt}  
  {\sl}  
  {\parindent}     
  {\bf}  
  {. }    
  { }    
  {}     
\theoremstyle{theorem}
\newtheorem{theorem}{Theorem}
\newtheorem{corollary}[theorem]{Corollary}
\newtheorem{remark}[theorem]{Remark}
\newtheorem{proposition}[theorem]{Proposition}
\newtheorem{lemma}[theorem]{Lemma}

\linethickness{0.5pt}

\newcommand{\bZ}{\mathbb{Z}}

\newcommand{\bN}{\mathbb{N}}

\newcommand{\aG}{\alpha}
\newcommand{\bG}{\beta}

\newcommand{\dG}{\delta}

\newcommand{\eG}{\varepsilon}

\newcommand{\sG}{\sigma}

\newcommand{\bds}{\begin{displaystyle}}
\newcommand{\eds}{\end{displaystyle}}

\title[Ramanujan-Nagell equations and perfect numbers.]{Ramanujan-Nagell type equations and perfect numbers.}

\author{Ph. Ellia - P. Menegatti}
\address{Dipartimento di Matematica, 35 via Machiavelli, 44100 Ferrara}
\email{phe@unife.it , paolo.menegatti@student.unife.it}


\date{\today}

\begin{document}
\maketitle

\thispagestyle{empty}

\begin{abstract} We prove that if $\dG$ is a triangular number congruent to 3 mod.4, then the equation $x-y=\dG$ has a finite number of solutions with $x,y$ both perfect numbers. We outline a general approach to determine the exact number of solutions and show that there is none for $\dG =3,15$. 
\end{abstract}


\section*{Introduction}
An integer $n \in \bN$ is said to be perfect if $\sG (n)=2n$ where $\sG$ is the sum of divisors function. By results of Euler every even perfect number has the form $n = 2^{p-1}(2^p-1)$ where $2^p-1$ is prime, whereas every odd perfect number is of the form $n = q^{4b+1}.\prod p_i^{2a_i}$, $q,p_i$ distinct primes, $q \equiv 1\,\,(mod\,\,4)$; in particular if $n$ is an odd perfect number, then $n \equiv 1\,\,(mod\,\,4)$. It is still unknown if odd perfect numbers exist (for some recent results see for example \cite{N}, \cite{OR}).

In \cite{LP}, Luca and Pomerance have proved, \emph{assuming} the $abc$-conjecture, that the equation $x-y=\dG$ has a finite number of solutions with $x,y$ perfect, if $\dG$ is odd. Our interest in the distance between two perfect numbers comes from this result and the following obvious remark: if one could prove that an odd integer cannot be the distance between two perfect numbers, then it would follow that every perfect number is even.

From Touchard's theorem (\cite{Touch2}) it follows that an integer $\dG \equiv \pm 1\pmod{12}$ cannot be the distance between two perfect numbers.  In \cite{Ellia} it has been shown that there exist infinitely many odd (triangular) numbers ($\not \equiv \pm 1\pmod{12}$) which cannot be the distance between perfect numbers. In this note, by using results on generalized Ramanujan-Nagell equations, we prove that if $\dG$ is a triangular number $\equiv 3\pmod{4}$, then $x-y=\dG$ has a finite number of solutions with $x,y$ perfect numbers. We also outline a general approach to determine the exact number of solutions. For example we show that $\dG =3, 15$ cannot be the distance between two perfect numbers.  


\section{Ramanujan-Nagell equations and perfect numbers.}

Let $D_1, D_2 \in \bZ$ be non zero integers, then the equation (in $x,n$)
\begin{equation}
\label{eq:RN} 
D_1x^2+D_2 = 2^n
\end{equation} 
is a \emph{generalized Ramanujan-Nagell} equation. Recall the following result of Thue (see for istance \cite{Lj})
\begin{theorem}
\label{THUEfini} 

Let $a,b,c,d \in \bZ$ such that $ad\neq 0$, $b^2-4ac\neq 0$. Then the equation

\begin{equation}
ax^2+bx+c = dy^n
\end{equation}
has only a finite number of solutions in integers $x$ and $y$ when $n\geq3$
\end{theorem}
Applying this result to $D_1x^2+D_2 = dy^3$, $d=1,2,4$, we conclude that

\begin{corollary}
\label{C-fini}
For $n\geq 3$, equation (\ref{eq:RN}) has a finite number of solutions $(x,n)$.
\end{corollary}

 An odd perfect number $n$ is $\equiv 1\pmod{4}$, while an even one, $m$, is $\equiv 0\pmod{4}$ except if $m=6$. It follows, for odd $\dG$, that $\dG \equiv 1\pmod{4}$ if $\dG = n-m$ or $\dG \equiv 3\pmod{4}$ if $\dG = m-n$ or $\dG +6=n$. Using the result above we get:

\begin{theorem}
\label{T-dG=3 fini}
Let $\dG =b(b-1)/2$ be a triangular number such that $\dG \equiv 3\pmod{4}$. Then the equation $x-y = \dG$ has a finite number of solutions with $x,y$ both perfect.
\end{theorem}

\begin{proof} We may assume $m-n=\dG = b(b-1)/2$, with $m,n$ perfect numbers and $m =2^{p-1}(2^p-1)$. Then we have:
\begin{equation}
\label{eq:pf prop}
2n = (2^p-1+b)(2^p-b)
\end{equation}
Moreover by Euler's theorem $n = q^{4b+1}\prod p_i^{2a_i}$, $q,p_i$ distinct primes, $q \equiv 1\pmod{4}$.

Since $(2^p-1+b, 2^p-b) = (2^p-1+b, 2b-1)$, if a prime, $p$, divides both $A = 2^p-1+b$ and $B=2^p-b$, it must divide $2b-1$. In any case we can write $A = p^\eG.A'$, $\eG \in \{0,1\}$ and $p^{2\aG} \parallel A'$. Similarly $B = p^{e}.B'$, $e \in \{0,1\}$, $p^{2\bG}\parallel B'$. It turns out that $A$ or $B$ is of the form: $d$ times a square, where $d$ is a (square free) divisor of $2(2b-1)$. So $2^p = dC^2-b+1$ or $2^p = dD^2+b$. By Corollary \ref{C-fini} each equation $dx^2 + D_2 = 2^n$ ($D_2 = -b+1$ or $b$) has a finite number of solutions. Since $2(2b-1)$ has a finite number of divisors we are done.  
\end{proof}

\begin{remark} As far as $\dG = b(b-1)/2$ congruent to $3$ mod. $4$, in order to show that $\dG$ can't be the distance between two perfect numbers one has:\\
(1) to show that $\dG +6$ is not perfect.\\
(2) for any square free divisor $d$ of $2(2b-1)$ to solve the equations: $dx^2+D_2 = 2^n$ ($D_2 \in \{-b+1, b\}$ (see proof of Theorem \ref{T-dG=3 fini}). For any solution $(x,n)$ such that $n=p$ is prime, check if $2^p-1$ is prime. If it is, check if $2^{p-1}(2^p-1)-\dG$ is perfect.
\end{remark}

Since a great deal is known on the generalized Ramanujan-Nagell equations (see \cite{SaS} for a survey), in the practice, for a given $\dG$, the above procedure should allow to conclude (see also \cite{Mi}, \cite{PeW} for an algorithmic approach).
Sometimes it is possible to go faster, for example:

\begin{proposition} The equation $x-y = 15$ has no solutions with $x,y$ both perfect numbers.
\end{proposition}

\begin{proof} This is the case $b=6$ of Theorem \ref{T-dG=3 fini}. If the Euler prime, $q$, of $n$ divides both $A$ and $B$ it must divide $2b-1 =11$, so $q=11$ which is impossible since $q \equiv 1\pmod{4}$. If $q \mid B$, then $A =2^p+5 = dx^2$, with $d=1$ or $11$. Since $x$ is odd we get $5\equiv d\pmod{8}$ which is absurd. We conclude that $q\mid A$ and $B = 2^p-6=dx^2$, where $d\mid 22$. Reducing mod.3 we see that $d=1$ is impossible. Reducing mod.4 we exclude the cases $d=11, 22$. Finally it is easy to see that the unique solution of $2x^2+6=2^n$ is $(x,n)=(\pm 1,3)$.
\end{proof}

To conclude let us see another example, the case $\dG =3$ which is still open.

\begin{lemma}
\label{L-dG=3}
Assume $m,n$ are perfect numbers such that $m-n=3$. Then $m = 2^{p-1}(2^p-1)$ with $2^p-1$ prime and $2^p= 5u^2+3$ for some integer $u$.
\end{lemma}

\begin{proof} Since $\dG = b(b-1)/2$ with $b=3$ and since $\dG +6=9$ is not perfect, we see that $m$ is even and $n$ is odd. So $m=2^{p-1}(2^p-1)$ with $2^p-1$ prime. Moreover $n = (2^{p-1}+1)(2^p-3)$ (equation (\ref{eq:pf prop}) in the proof of Theorem \ref{T-dG=3 fini}). Also $M=g.c.d.(A,B) = 5$ or $1$ ($A=2^{p-1}+1 $, $B=2^p-3$).

If $M=1$, from Euler's theorem, $A$ or $B$ is a square. Since $A = (2^{(p-1)/2)})^2+1$, $A$ can't be a square. Since $B = 2^p-3 \equiv 2\pmod{3}$, $B$ can't be a square. It follows that $M=5$.

Since $n=m-3$ and $m\equiv 1\pmod{3}$, we get $n \equiv 1\pmod{3}$. We also have $n=qD^2$ ($q\equiv 1\pmod{4}$, the Euler's prime). It follows that $q \equiv 1\pmod{3}$. In particular $q \neq 5$. Finally if $q \nmid A$, then $A = 2^{p-1}+1 = 5C^2$. Since $C$ is odd $C^2\equiv 1\pmod{8}$. Since we may assume $p\geq 4$, we get a contradiction. So $q \mid A$ and $B = 2^p-3 = 5u^2$ for some integer $u$.
\end{proof}

\begin{corollary}
If $\dG = |x-y|$, with $x,y$ perfect numbers, then $\dG > 3$.
\end{corollary}

\begin{proof} The cases $\dG \leq 2$ follow from considerations on congruences. If $\dG =3$, then from Lemma \ref{L-dG=3}: $m-n=3$, where $m=2^{p-1}(2^p-1)$ and $2^p=5u^2+3$. So $(u,p)$ is a solution of the equation $5x^2 + 3 = 2^n$. It is known (\cite{Le}) that the only solutions in positive integers of this equations are $(x,n) = (1,3), (5,7)$. Since $25$ and $2^6(2^7-1)-3 = 8125$ are not perfect numbers, we conclude.
\end{proof}

\medskip

If $\dG$ is not a triangular number congruent to $3\pmod{4}$ we don't have any longer the factorization $2n = (2^p-1+b)(2^p-b)$ and things get harder. The cases $\dG = 5,7$ can be excluded by congruences considerations. However for odd $\dG \leq 15$, the case $\dG=9$ is still open, as is still open the problem to show that a triangular number $\equiv 3\pmod{4}$ can't be the distance between two perfect numbers.




\begin{thebibliography}{RN-PerfectNbs}


\bibitem{Ellia} Ellia, Ph.: {\it On the distance between perfect numbers}, J. of Combinatorics and Number Theory, vol.4, n.2, 105-108 (2012) 



\bibitem{Touch2} Holdener, J.A.: {\it A theorem of Touchard on the form of odd perfect numbers}, Amer. Math. Monthly 109, 661-663 (2002)

\bibitem{Le} Le, M.: {\it On the diophantine equation $D_1x^2+D_2=2^{n+2}$}, Acta Arith. {\bf 64}, 29-41 (1993)

\bibitem{Lj} Ljunggren, W.: {\it On the diophantine equation $Cx^2+D=2y^{n}$}, Math. Scand. {\bf 18}, 69-86 (1966)


\bibitem{LP} Luca, F.-Pomerance, C.: {\it On the radical of a perfect number}, New York J. Math. {\bf 16}, 23-30 (2010)


\bibitem{Mi} Mignotte, M.: {\it On the automatic resolution of certain diophantine equation} in EUROSAM 8, Proceedings, Lecture Notes in Comput. Sci.  {\bf 174}, 378-385 (1984)

\bibitem{N} Nielsen Pace, P.:{\it Odd perfect numbers have at least nine distinct prime factors}, Math. Comp., {\bf 76}, 2109-2126 (2007)

\bibitem{OR} Ochem, P.-Rao, M.: {\it Odd perfect numbers are greater than $10^{1500}$}, Math. Comp., 81, 1869-1877 (2012)

\bibitem{PeW} Petho, A.-de Weger, B.M.M.: {\it Products of prime powers in binary recurrence sequences I: The hyperbolic case, with ana application to the generalized Ramanujan-Nagell equation}, Math. Comp. {\bf 47}, 713-727 (1986)

\bibitem{SaS} Saradha, N.-Srinivasan, A.: {\it Generalized Lebesgue-Ramanujan-Nagell Equations}, in Saradha N. {\it Diophantine Equations}, Narosa Publishing House, New Delhi, India,  207-223 (2007)


\end{thebibliography}
\end{document}